\documentclass[a4, 12pt, reqno]{article}
\usepackage{amscd, amsmath, amssymb, amsfonts, amsthm}
\usepackage[arrow, matrix]{xy}
\usepackage{graphicx}
\usepackage{color}
\usepackage[utf8]{inputenc}
\usepackage[english]{babel}

\theoremstyle{plain}
\newtheorem{Theorem}{Theorem}[section]
\newtheorem{Lemma}{Lemma}[section]

\theoremstyle{definition}

\newtheorem{Conjecture}{Conjecture}

\numberwithin{equation}{section}

\begin{document}

\title{Disproof of a conjecture on the edge Mostar index}

\author{Fazal Hayat$^{a}$,  Shou-Jun Xu$^{a}$\footnote{Corresponding author
 \newline E-mail addresses: fhayatmaths@gmail.com (F. Hayat), shjxu@lzu.edu.cn (S.-J. Xu), zhoubo@scnu.edu.cn (B. Zhou)}, Bo Zhou$^{b}$ \\
$^a$School of Mathematics and Statistics,  Gansu Center for Applied Mathematics,\\
 Lanzhou University,  Lanzhou 730000,  P.R. China \\
$^{b}$School of Mathematical Sciences,  South China Normal University, \\
 Guangzhou 510631, P.R. China}

 \date{}
\maketitle

\begin{abstract}
For a given connected graph $G$, the edge Mostar index $Mo_e(G)$ is defined as  $Mo_e(G)=\sum_{e=uv \in E(G)}|m_u(e|G) - m_v(e|G)|$,  where $m_u(e|G)$ and $m_v(e|G)$ are respectively, the number of edges of $G$ lying closer to  vertex $u$ than to vertex $v$ and the number of edges of $G$ lying closer to vertex $v$ than to vertex $u$. We determine a sharp upper bound for the edge Mostar index on bicyclic graphs and identify the graphs that attain the bound, which disproves a conjecture  proposed by Liu et al. [Iranian J. Math. Chem. 11(2) (2020) 95--106].   
\end{abstract}

\section{Introduction}
A molecular graph is a simple graph such that its vertices correspond to the atoms and the edges to the bonds of a molecule.  Let $G$ be a graph with vertex set $V(G)$ and edge set $E(G)$.  A topological index of $G$ is a real number related to $G$.  They are widely used for characterizing molecular graphs, establishing relationships between the structure and properties of molecules, predicting the biological activity of chemical compounds, and making their chemical applications.

Dosli\'c et al. \cite{DoM} introduced a bond-additive structural invariant as a quantitative refinement of the distance non-balancedness and also a measure of peripherality in graphs, named the Mostar index. For a  graph $G$, the Mostar index is defined as
\[
 Mo(G)=\sum_{e=uv \in E(G)}|n_u(e|G) - n_v(e|G)|.
\]
where $n_u(e|G)$ is the number of vertices of $G$ closer to $u$ than to $v$ and $n_v(e|G)$ is the number of vertices closer to $v$ than to $u$.

Since its introduction in 2018, the Mostar index has already incited a lot of research, mostly concerning  trees \cite{AXK, DL, DL1, DL2, GXD, HX1, HZ1}, unicyclic graphs \cite{LD}, bicyclic graphs \cite{Te}, tricyclic graphs \cite{HX}, cacti \cite{ HZ}, chemical graphs \cite{ DL3, HLM, XZT, XZT2}.

Arockiaraj et al. \cite{ACT}, introduced the edge Mostar index as a quantitative refinement of the distance non-balancedness, also measure the peripherality of every edge and consider the contributions of all edges into a global measure of peripherality for a given chemical graph.  The edge Mostar index of $G$ is  defined as
\[
 Mo_e(G)=\sum_{e=uv \in E(G)}\psi_G(uv),
\]
where $\psi_G(uv)=|m_u(e|G) - m_v(e|G)|$, $m_u(e|G)$ and $m_v(e|G)$ are respectively, the number of edges of $G$ lying closer to  vertex $u$ than to vertex $v$ and the number of edges of $G$ lying closer to vertex $v$ than to vertex $u$. We use $\psi(uv)=|m_u(e) - m_v(e)|$ for short, if there is no ambiguity.

Imran et al. \cite{IAI} studied the edge Mostar index of chemical structures and nanostructures using graph operations. Havare \cite{H} computed the edge Mostar index for several classes of cycle-containing graphs.
Liu et al. \cite{LSX}  determined the extremal values of the edge Mostar index among trees, unicyclic graphs, cacti and posed two conjectures for the extremal  edge Mostar index among bicyclic graphs.  Ghalavand et al. \cite{GAN} gave proof to a conjecture in \cite{LSX}, and obtained the graph that minimizes the edge Mostar index among bicyclic graphs.

Motivated directly by \cite{LSX} and \cite{GAN}, we determine the unique graph that maximizes the edge Mostar index over all bicyclic graphs with  fixed size, and disprove the following conjecture.

\begin{Conjecture}\cite{LSX}\label{cnj}
If the size $m$ of bicyclic graphs is large enough, then the bicyclic graphs $B_m$ (see Fig \ref{x}) and $B^5_m$ (see Fig \ref{y}) have the maximum edge Mostar index.
\end{Conjecture}

We disprove the Conjecture \ref{cnj}, by giving the following result.

\begin{Theorem}\label{a} Let $G$ be a bicyclic graph of size $m \geq 5$. Then
$$Mo_e(G) \leq \left\{\begin{array}{ll}
        4,         & \hbox{if $m=5$, and equality holds iff $G \cong B^3_m, B^4_m$}; \\
        m^2-3m-6,  & \hbox{if $6 \leq m \leq 8$, and equality holds iff $G \cong B^1_m, B^3_m$};  \\
        48,        & \hbox{if $m=9$, and equality holds iff $G \cong B_m, B^1_m, B^2_m, B^3_m, B^4_m$};  \\
        m^2-m-24,  & \hbox{if $m \geq 10$, and equality holds iff $G \cong B_m$};
\end{array}\right.$$
(where $B_m, B^1_m, B^2_m$ and  $B^3_m, B^4_m$ are depicted in Fig \ref{x} and Fig \ref{y}, respectively).
\end{Theorem}

In section 2, we give some definitions and preliminary results.  Theorem \ref{a} is proved in section 3.

\section{Preliminaries}

 \begin{figure}
\centering
\includegraphics[width=11cm,height=3cm]{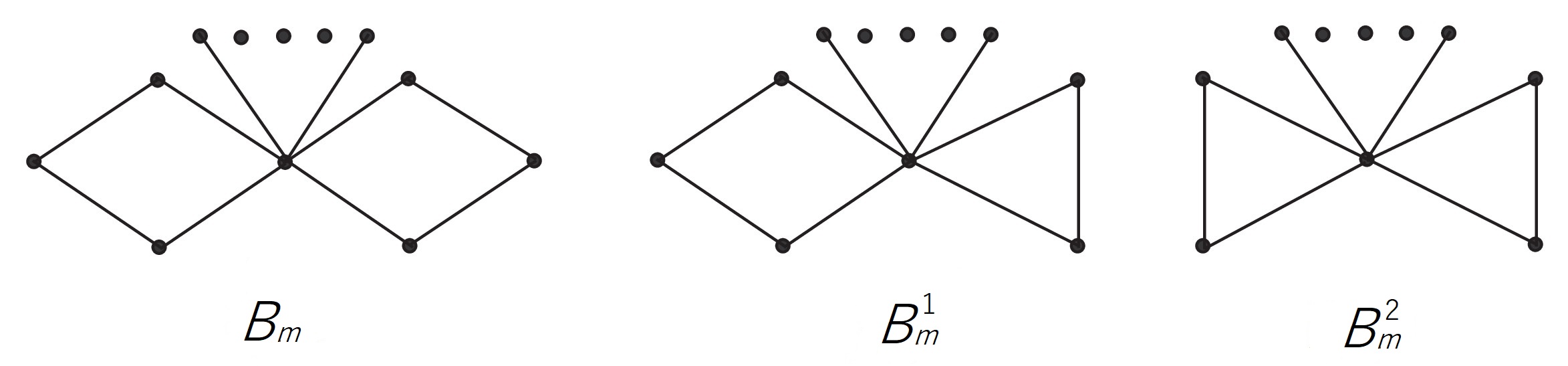}
\caption{ Graphs $B_m, B^1_m$ and $B^2_m$.}
\label{x}
\end{figure}

 \begin{figure}
\centering
\includegraphics[width=12cm,height=3.5cm]{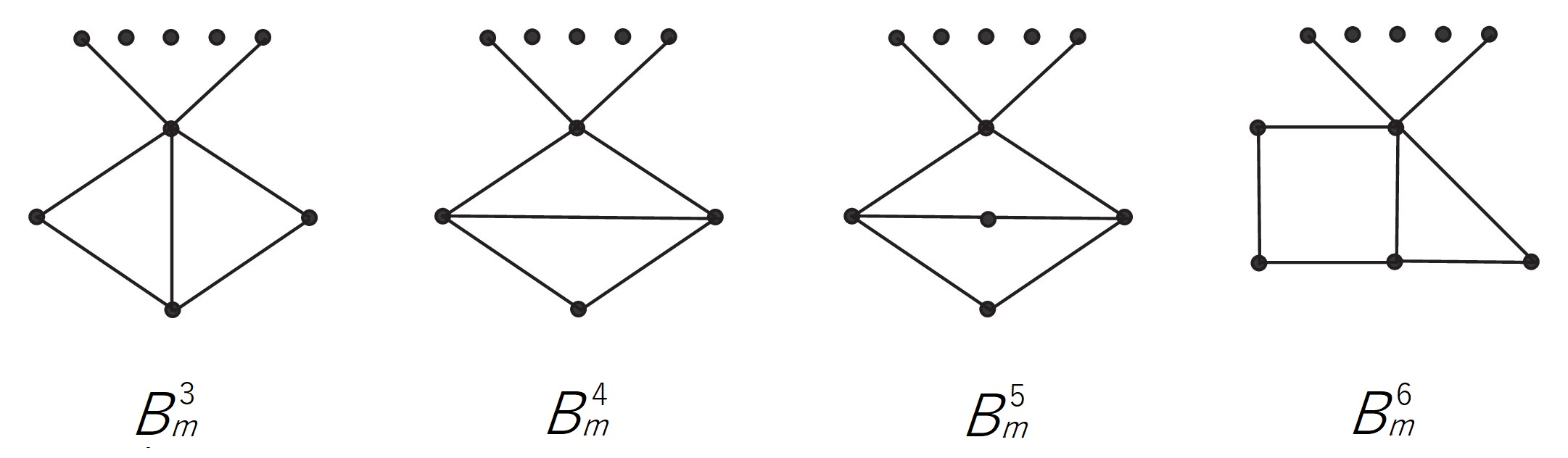}
\caption{Graphs $B^3_m, B^4_m, B^5_m$ and $B^6_m$.   }
\label{y}
\end{figure}

The order and size of $G$ is the cardinality of $V(G)$ and $E(G)$, respectively. For  $v \in V(G)$, let $N_G(v)$  be the set of vertices that are adjacent to $v$ in $G$. The degree of $v \in V(G)$, denoted by $d_G(v)$,  is the cardinality of  $N_G(v)$. A vertex with degree one is called a pendent vertex and an edge incident to a pendent vertex is called a pendent edge. The distance between $u$ and $v$ in $G$ is the least length of the path connecting $u$ and $v$ denoted by $d_G(u,v)$.
A graph $G$ with $n$ vertices is a bicyclic graph if $|E(G)|=n+1$. As usual, by $S_n$, $P_n$ and $C_n$ we denote the star, path and cycle on $n$ vertices, respectively.

The join of two graphs $G_1$ and $G_2$ is denoted by $G_1 \cdot G_2$, is obtained by identifying one vertex of $G_1$ and $G_2$. Set $u$ as the identified vertex of $G_1$ and $G_2$. If $G_1$ contains a cycle and $u$ belongs to some cycle, and $G_2$ is a tree, then we call $G_2$ a pendent tree in $G_1 \cdot G_2$  associated with $u$.  For each $e \in E(G_1)$, every path from $e$ to some edges of $G_2$ pass through $u$. Therefore, the contribution of $G_2$ to  $\sum_{e\in E(G_1)}\psi(e)$ totally depends on the size of $G_2$,  that is, changing the structure of $G_2$ cannot alter the value of $\sum_{e\in E(G_1)}\psi(e)$.

\begin{Lemma}\cite{LSX}\label{12}
Let $G$ be a  graph of size $m$ with a cut edge $e=uv$.  Then $\psi(e)\leq m-1$  with equality if and only if  $e=uv$ is a pendent edge.
\end{Lemma}

By Lemma \ref{12}, if $e$ is a pendent edge, then $\psi(e)$ is maximum. Therefore, it is easy to verify the following result.

\begin{Lemma}\label{13}
Let $H$ be a graph of size $m$.  Then $$Mo_e(H_1 \cdot H)\leq Mo_e(H_1 \cdot S_m),$$  where the common vertex of $H_1$ and $ S_m$  is the center
of $S_m$.
\end{Lemma}

Let $S_{m,r} = S_{m-r} \cdot C_r$, where the common vertex of $S_{m-r}$ and $C_r$ is the center  of $S_{m-r}$.

\begin{Lemma}\cite{LSX}\label{1} Let $G$ be a unicyclic graph of size $m \geq 3$. Then
$$Mo_e(G) \leq \left\{\begin{array}{ll}
        m^2-2m-3,  & \hbox{if $3 \leq m \leq 8$, and equality holds iff $G \cong S_{m, 3} $};  \\
        60,        & \hbox{if $m=9$, and equality holds iff $G \cong S_{m, 3}, S_{m, 4} $};  \\
        m^2-m-12,  & \hbox{if $m \geq 10$, and equality holds iff $G \cong S_{m, 4} $}.
\end{array}\right.$$
\end{Lemma}

\begin{Lemma}\label{14} Let $G_1$ be a connected graph of size $m_1$  and  $G_2$ be a unicyclic graph of size $m_2$. Then
$$ Mo_e(G_1 \cdot G_2 ) \leq \left\{\begin{array}{ll}
         Mo_e(G_1 \cdot S_{m_2, 3} )  & \hbox{for $m_1 + m_2 \leq 8$};   \\
         Mo_e(G_1 \cdot S_{m_2, 3} )= Mo_e(G_1 \cdot S_{m_2, 4} )       & \hbox{for $m_1 + m_2 = 9$};  \\
        Mo_e(G_1 \cdot S_{m_2, 4} )  & \hbox{for $m_1 + m_2 \geq 10$};
\end{array}\right.$$
where the common vertex of $G_1$ and $S_{m_2, 3}$ (resp. $G_1$ and $S_{m_2, 4}$)  is the center
of $S_{m_2, 3}$ (resp. $S_{m_2, 4}$).
\end{Lemma}
\begin{proof} If $ m_1 + m_2 \leq 8 $, then by Lemma \ref{1}, we have
\begin{eqnarray*}
 Mo_e(G_1 \cdot G_2 ) & = & \sum_{e=uv \in E(G_1 \cdot G_2)}\psi_{G_1 \cdot G_2}(uv) \\
          & = & \sum_{e=uv \in E(G_1)}\psi_{G_1 \cdot G_2}(uv) + \sum_{e=uv \in E( G_2)}\psi_{G_1 \cdot G_2}(uv) \\
          & = & \sum_{e=uv \in E(G_1)}\psi_{G_1 \cdot G_2}(uv) +  Mo_e(S_{m_1} \cdot G_2 )-m_1(m-1) \\
          &\leq& \sum_{e=uv \in E(G_1)}\psi_{G_1 \cdot S_{m_2, 3}}(uv) +  Mo_e(S_{m_1} \cdot S_{m_2, 3})-m_1(m-1) \\
          & = & \sum_{e=uv \in E(G_1)}\psi_{G_1 \cdot S_{m_2, 3}}(uv) + \sum_{e=uv \in E( S_{m_2, 3})}\psi_{G_1 \cdot S_{m_2, 3}}(uv) \\
          & = &  Mo_e(G_1 \cdot S_{m_2, 3} ).
 \end{eqnarray*}
Similarly, if $ m_1 + m_2 = 9 $, then $Mo_e(G_1 \cdot G_2 ) \leq Mo_e(G_1 \cdot S_{m_2, 3} )= Mo_e(G_1 \cdot S_{m_2, 4} )$;  if  $ m_1 + m_2 \geq 10 $, then  $Mo_e(G_1 \cdot G_2 ) \leq  Mo_e(G_1 \cdot S_{m_2, 4} )$.
\end{proof}

The $\Theta$-graph is the graph connecting two fixed vertices by three internally disjoint paths. If the corresponding three paths have the lengths $l_1, l_2, l_3$ with $l_1 \leq l_2 \leq l_3$,  then denote as $\Theta (l_1, l_2, l_3)$. Let $\mathcal{G}_m^1$ be the set of bicyclic graphs of size $m$ having exactly two edge-disjoint cycles. Let $\mathcal{G}_m^2$ be the set of bicyclic graphs of size $m$ having three cycles. Moreover, if $G \in \mathcal{G}_m^2$, then it has a $\Theta$-graph as its subgraph, which is called the brace of $G$.

\section{Proof of Theorem \ref{a}}

\begin{Lemma}\label{31} Let $G \in \mathcal{G}_m^1$ . Then
$$Mo_e(G) \leq \left\{\begin{array}{ll}
        m^2-3m-6,  & \hbox{if $6 \leq m \leq 8$, and equality holds iff $G \cong B^2_m$};  \\
        48,        & \hbox{if $ m=9$, and equality holds iff $G \cong B_m, B^1_m, B^2_m$};  \\
        m^2-m-24,  & \hbox{if $m \geq 10$, and equality holds iff $G \cong B_m$}.
\end{array}\right.$$
\end{Lemma}
\begin{proof}
Let $G \in \mathcal{G}_m^1$. Then there are two unicyclic graphs $G_1$ and $G_2$ with size $m_1$ and $m_2$, respectively such that $G = G_1 \cdot G_2 $. Then, in view of Lemma \ref{14}. If $6 \leq m \leq 8$,  we get
\begin{eqnarray*}
 Mo_e(G ) & = & Mo_e(G_1 \cdot G_2 ) \leq  Mo_e(G_1 \cdot S_{m_2, 3} )\\
          &\leq & Mo_e(S_{m_1, 3} \cdot S_{m_2, 3} )= Mo_e(B^2_m);
\end{eqnarray*}
if $m= 9$, then
\begin{eqnarray*}
 Mo_e(G ) & = & Mo_e(G_1 \cdot G_2 ) \leq  Mo_e(G_1 \cdot S_{m_2, 3} )= Mo_e(G_1 \cdot S_{m_2, 4} )\\
          &\leq & Mo_e(S_{m_1, 3} \cdot S_{m_2, 3} )=  Mo_e(S_{m_1, 3} \cdot S_{m_2, 4} )= Mo_e(B^2_m)\\
           & = &  Mo_e(S_{m_1, 4} \cdot S_{m_2, 4} ) = Mo_e(B_m) = Mo_e(B^1_m)= Mo_e(B^2_m);
\end{eqnarray*}
If $m \geq 10$, we have
\begin{eqnarray*}
 Mo_e(G ) & = & Mo_e(G_1 \cdot G_2 ) \leq  Mo_e(G_1 \cdot S_{m_2, 4} )\\
          &\leq & Mo_e(S_{m_1, 4} \cdot S_{m_2, 4} )= Mo_e(B_m).
\end{eqnarray*}
By simple calculation, we get $Mo_e(B_m)= m^2-m-24$, $Mo_e(B^1_m)= m^2-2m-15$, and $Mo_e(B^2_m)= m^2-3m-6$. This completes the proof.
\end{proof}

\begin{Lemma}\label{32} Let $G \in \mathcal{G}_m^2$ with brace $\Theta (1,2,2)$. Then
$$Mo_e(G) \leq \left\{\begin{array}{ll}
        4,  & \hbox{if $m=5$, and equality holds iff $G \cong B^3_m, B^4_m$};  \\
        m^2-3m-6,  & \hbox{if $6 \leq m \leq 8$, and equality holds iff $G \cong B^3_m$};  \\
        48,        & \hbox{if $ m=9$, and equality holds iff $G \cong  B^3_m, B^4_m$};  \\
        m^2-2m-15,  & \hbox{if $m \geq 10$, and equality holds iff $G \cong B^4_m$}.
\end{array}\right.$$
\end{Lemma}
\begin{proof}
Let $G \in \mathcal{G}_m^2$ with brace $\Theta (1,2,2)$. Suppose that $v_1, v_2, v_3, v_4 $ be the vertices in $\Theta (1,2,2)$ of $G$ with $d_G(v_1)=d_G(v_2)=3$, and  $d_G(v_3)=d_G(v_4)=2$. Let $a_i$ be the number of pendent edges of $v_i$ ($i=1,2,3,4$). Suppose that $a_1 \geq a_2$ and $a_3 \geq a_4$. Let $G_1$ be the graph obtained from $G$ by shifting $a_2$ pendent edges from $v_2$ to $v_1$. We have
\begin{eqnarray*}
&& Mo_e(G_1 )-  Mo_e(G )  =  (a_1+a_2+2-2)-(a_1+2-a_2-2)\\
                       &\qquad + & (a_1+a_2+a_4+2-a_3-1)-(a_1+a_4+2-a_3-1)+(a_1+a_2+a_3+2-a_4-1)\\
                       &\qquad - & (a_1+a_3+2-a_4-1)+(a_4+2-a_3-1)-(a_2+a_4+2-a_3-1)\\
                       &\qquad + & (a_3+2-a_4-1)-(a_2+a_3+2-a_4-1)\\
                       & = & 2a_2 \geq 0.
\end{eqnarray*}
Hence, $ Mo_e(G_1 ) \geq  Mo_e(G )$ and equality holds if and only if $a_2 =0$, i.e., $G \cong G_1$.

Let $G_2$ be the graph obtained from $G_1$ by shifting $a_4$ pendent edges from $v_4$ to $v_3$. We have
\begin{eqnarray*}
&& Mo_e(G_2 )- Mo_e(G_1) =  (a_1+2-2)-(a_1+2-2)+(a_3+a_4+1-a_1-2)\\
                       &\qquad - & (a_3+1-a_1-a_4-2)+(a_1+a_3+a_4+2-1)-(a_1+a_3+2-a_4-1)\\
                       &\qquad + & (a_3+a_4+1-2)-(a_3+1-a_4-2)+(a_3+a_4+2-1)-(a_3+2-a_4-1)\\
                       & = & 8a_4 \geq 0.
\end{eqnarray*}
Hence, $ Mo_e(G_2 ) \geq  Mo_e(G_1 )$ and equality holds if and only if $a_4 =0$, i.e., $G_2 \cong G_1$.

Let $G_3$ be the graph obtained from $G_2$ by shifting $a_1$ pendent edges from $v_1$ to $v_3$. We have
\begin{eqnarray*}
&& Mo_e(G_3 )- Mo_e(G_2) =  -(a_1+2-2)+(a_1+a_3+1-2)-(a_3+1-a_1-2)\\
                       &\qquad + & (a_1+a_3+2-1)-(a_1+a_3+2-1)+(a_1+a_3+1-2)-(a_3+1-2)\\
                       &\qquad + & (a_1+a_3+2-1)-(a_3+2-1)\\
                       & = & 3a_1 \geq 0.
\end{eqnarray*}
Hence, $ Mo_e(G_3 ) \geq  Mo_e(G_2 )$ and equality holds if and only if $a_1 =0$, i.e., $G_3 \cong G_2$. Clearly, $G_2 \cong B^3_m$ with $a_3= 0$, and $G_3 \cong B^4_m$. Also, by simple calculation, we have $Mo_e(B^3_m)= m^2-3m-6$, $Mo_e(B^4_m)= m^2-2m-15$.
\end{proof}

\begin{Lemma}\label{33} Let $G \in \mathcal{G}_m^2$ with brace $\Theta (2,2,2)$. If $m \geq 6$, then $Mo_e(G) \leq m^2-m-28$ with equality if and only if $G \cong B^5_m$ (see Fig \ref{y}).
\end{Lemma}
\begin{proof}
Let $G \in \mathcal{G}_m^2$ with brace $\Theta (2,2,2)$. Let $v_1, v_2, v_3, v_4, v_5 $ be vertices in $\Theta (2,2,2)$ of $G$ with $d_G(v_1)=d_G(v_2)=3$ and  $d_G(v_3)=d_G(v_4)=d_G(v_5)=2$. Let $a_i$ be the number of pendent edges of $v_i$ ($i=1,2,3,4,5$). Suppose that $a_1 \geq a_2$ and $a_3 \geq a_4 \geq a_5$. Let $G_1$ be the graph obtained from $G$ by shifting $a_2$ pendent edges from $v_2$ to $v_1$. We have
\begin{eqnarray*}
 &&Mo_e(G_1 )-  Mo_e(G ) =  (a_1+a_2+a_4+a_5+2-a_3-1)\\
                       &\qquad - & (a_1+a_4+a_5+2-a_2-a_3-1)+(a_1+a_2+a_3+a_5+2-a_4-1)\\
                       &\qquad - & (a_1+a_3+a_5+2-a_2-a_4-1)+(a_1+a_2+a_3+a_4+2-a_5-1)\\
                       &\qquad - & (a_1+a_3+a_4+2-a_2-a_5-1)+(a_1+a_2+a_3+1-a_4-a_5-2)\\
                       &\qquad - & (a_1+a_3+1-a_2-a_4-a_5-2)+(a_1+a_2+a_4+1-a_3-a_5-2)\\
                       &\qquad - & (a_1+a_4+1-a_2-a_3-a_5-2)+(a_1+a_2+a_5+1-a_3-a_4-2)\\
                       &\qquad - & (a_1+a_5+1-a_2-a_3-a_4-2)\\
                       & = & 12a_2 \geq 0.
\end{eqnarray*}
Hence, $ Mo_e(G_1 ) \geq  Mo_e(G )$ and equality holds if and only if $a_2 =0$, i.e., $G \cong G_1$.

Let $G_2$ be the graph obtained from $G_1$ by shifting $a_5$ pendent edges from $v_5$ to $v_3$. We get
\begin{eqnarray*}
&& Mo_e(G_2 )-  Mo_e(G_1 ) =  (a_3+a_5+1-a_1-a_4-2)-(a_3+1-a_1-a_4-a_5-2)\\
                       &\qquad + & (a_1+a_3+a_5+2-a_4-1)-(a_1+a_3+a_5+2-a_4-1)\\
                       &\qquad + & (a_1+a_3+a_4+a_5+2-1)-(a_1+a_3+a_4+2-a_5-1)\\
                       &\qquad + & (a_1+a_3+a_5+1-a_4-2)-(a_1+a_3+1-a_4-a_5-2)\\
                       &\qquad + & (a_3+a_5+2-a_1-a_4-1)-(a_3+a_5+2-a_1-a_4-1)\\
                       &\qquad + & (a_3+a_4+a_5+2-a_1-1)-(a_3+a_4+2-a_1-a_5-1)\\
                       & = & 8a_5 \geq 0.
\end{eqnarray*}
Hence, $ Mo_e(G_2 ) \geq  Mo_e(G_1 )$ and equality holds if and only if $a_5 =0$, i.e., $G_2 \cong G_1$.

Let $G_3$ be the graph obtained from $G_2$ by shifting $a_4$ pendent edges from $v_4$ to $v_3$. By symmetry, $ Mo_e(G_3 ) \geq  Mo_e(G_2 )$, and equality holds if and only if  $G_3 \cong G_2$.
Let $G_4$ be the graph obtained from $G_3$ by shifting $a_1$ pendent edges from $v_1$ to $v_3$. We get
\begin{eqnarray*}
&& Mo_e(G_4 )-  Mo_e(G_3)=  (a_1+a_3+1-2)-(a_3+1-a_1-2)\\
                       &\qquad + & (a_1+a_3+2-1)-(a_1+a_3+2-1)+ (a_1+a_3+2-1)-(a_1+a_3+2-1)\\
                       &\qquad +&(a_1+a_3+1-2)-(a_1+a_3+1-2)+(a_1+a_3+2-1)-(a_3+2-a_1-1)\\
                       &\qquad + & (a_1+a_3+2-1)-(a_3+2-a_1-1)\\
                       & = & 6a_1 \geq 0.
\end{eqnarray*}
Hence, $ Mo_e(G_4 ) \geq  Mo_e(G_3 )$ and equality holds if and only if $a_1 =0$, i.e., $G_4 \cong G_3$. Note that $G_4 \cong B^5_m$. Clearly, $Mo_e(B^5_m )=m^2-m-28$.
\end{proof}

\begin{Lemma}\label{34} Let $G \in \mathcal{G}_m^2$ with brace $\Theta (1,2,3)$. If $m \geq 6$, then $Mo_e(G) \leq m^2-2m-16$ with equality if and only if $G \cong B^6_m$ (see Fig \ref{y}).
\end{Lemma}
\begin{proof}
Let $G \in \mathcal{G}_m^2$ with brace $\Theta (1,2,3)$. Let $v_1, v_2, v_3, v_4, v_5 $ be vertices in $\Theta (1,2,3)$ of $G$ with $d_G(v_1)=d_G(v_2)=3$ and  $d_G(v_3)=d_G(v_4)=d_G(v_5)=2$. Let $a_i$ be the number of pendent edges of $v_i$ ($i=1,2,3,4,5$). Suppose that $a_1 + a_4 \geq a_2 + a_5$ . Let $G_1$ be the graph obtained from $G$ by shifting $a_2$ (resp. $a_5$) pendent edges from $v_2$ (resp. $v_5$) to $v_1$ (resp. $v_4$). We deduce that
\begin{eqnarray*}
&& Mo_e(G_1 )-  Mo_e(G ) =  (a_1+a_2+a_4+a_5+2-2)-(a_1+a_4+2-a_2-a_5-2)\\
                       &\qquad + & (a_1+a_2+a_4+a_5+3-a_3-1)-(a_1+a_4+3-a_3-1)+(3-a_3-1)\\
                       &\qquad - & (a_2+a_5+3-a_3-1)+(a_1+a_2+a_3+3-a_4-a_5-1)\\
                       &\qquad - & (a_1+a_2+a_3+3-a_4-a_5-1)+(a_1+a_2+a_3+3-a_4-a_5-1)\\
                       &\qquad - & (a_1+a_2+a_3+3-a_4-a_5-1)+(a_1+a_2+a_4+a_5+2-2)\\
                       &\qquad - & (a_1+a_4+2-a_2-a_5-2)\\
                       & = & 4( a_2 + a_5 ) \geq 0.
\end{eqnarray*}
Hence, $ Mo_e(G_1 ) \geq  Mo_e(G )$ and equality holds if and only if $a_2= a_5 =0$, i.e., $G \cong G_1$.

Let $G_2$ be the graph obtained from $G_1$ by shifting $a_4$ pendent edges from $v_4$ to $v_1$. We obtain
\begin{eqnarray*}
&& Mo_e(G_2 )- Mo_e(G_1)=  (a_1+a_4+2-2)-(a_1+a_4+2-2)\\
                       &\qquad + & (a_1+a_4+3-a_3-1)-(a_1+a_4+3-a_3-1)+(3-a_3-1)\\
                       &\qquad - & (3-a_3-1)+(a_1+a_3+a_4+3-1)-(a_1+a_3+3-a_4-1)\\
                       &\qquad + & (a_1+a_3+a_4+3-1)-(a_1+a_3+3-a_4-1)\\
                       &\qquad + & (a_1+a_4+2-2)-(a_1+a_4+2-2)\\
                       & = & 4 a_4 + a_1 -a_3 \geq 0.
\end{eqnarray*}
Hence, $ Mo_e(G_2 ) \geq  Mo_e(G_1 )$ and equality holds if and only if $a_4 =0, a_1= a_3$, i.e., $G_2 \cong G_1$.

Let $G_3$ be the graph obtained from $G_2$ by shifting $a_3$ pendent edges from $v_3$ to $v_1$. We have
\begin{eqnarray*}
&& Mo_e(G_3 )- Mo_e(G_2)=  (a_1+a_3+2-2)-(a_1+2-2)+(a_1+a_3+3-1)\\
                       &\qquad - & (a_1+3-a_3-1)+(a_1+a_3+3-1)-(a_1+a_3+3-1)+(a_1+a_3+2-2)\\
                       &\qquad - & (a_1+2-2)+(3-1)-(3-a_3-1)+(a_1+a_3+3-1)-(a_1+a_3+3-1)\\
                       & = & 5a_3 \geq 0.
\end{eqnarray*}
Hence, $ Mo_e(G_3 ) \geq  Mo_e(G_2 )$ and equality holds if and only if $a_3 =0$, i.e., $G_3 \cong G_2$. Note that $G_3 \cong B^6_m$. Clearly, $Mo_e(B^6_m )=m^2-2m-16$.
\end{proof}

\begin{Lemma}\label{35} Let $G \in \mathcal{G}_m^2$ with brace $\Theta (a,b,c)$. Then
$$Mo_e(G) \leq \left\{\begin{array}{ll}
        4,  & \hbox{if $m=5$, and equality holds iff $G \cong B^3_m, B^4_m$};  \\
        m^2-3m-6,  & \hbox{if $6 \leq m \leq 8$, and equality holds iff $G \cong B^3_m$};  \\
        48,        & \hbox{if $ m=9$, and equality holds iff $G \cong  B^3_m, B^4_m$};  \\
        m^2-2m-15,  & \hbox{if $10 \leq m \leq 12$, and equality holds iff $G \cong B^4_m$}; \\
        128,        & \hbox{if $ m=13$, and equality holds iff $G \cong  B^4_m, B^5_m$};  \\
        m^2-m-28,  & \hbox{if $ m \geq 14$, and equality holds iff $G \cong B^5_m$.}
\end{array}\right.$$
\end{Lemma}
\begin{proof}
 Suppose that $x$ and $y$ are the vertices with degree 3 in $\Theta (a,b,c)$ of $G$. Let $P_{a+1}, P_{b+1}, P_{c+1}$ be the three disjoint paths connecting $x$ and $y$. We proceed with the following three possible cases.

\noindent {\bf Case 1.} $c \geq b \geq a \geq 3$.

\noindent {\bf Subcase 1.1.} $c = b = a \geq 3$.

We choose  six edges such that each one is incident to $x$ or $y$ in $\Theta (a,b,c)$. Let $e$ be one of the six edges. Then $\psi(e) \leq m-7$. This fact is also true for the remaining five edges. Thus,
\[
Mo_e(G) \leq 6(m-7)+(m-6)(m-1) < m^2-m-28.
\]

\noindent {\bf Subcase 1.2.} $c \geq b \geq 4, a = 3$.

Let  $e_1$ be one of the four edges incident to $x$ or $y$ in the paths $P_{b+1}$ and $P_{c+1}$, and  $e_2$ be one of the two edges incident to $x$ or $y$ in the path $P_{a+1}$. Then $\psi(e_1) \leq m-8$, and $\psi(e_2) \leq m-9$. Hence,
\[
Mo_e(G) \leq 4(m-8)+2(m-9)+(m-6)(m-1) < m^2-m-28.
\]

\noindent {\bf Case 2.} $c \geq b \geq a =2$.

\noindent {\bf Subcase 2.1.} $c = b = a =2$.

The Subcase follows from Lemma \ref{33}.

\noindent {\bf Subcase 2.2.} $c \geq 3, b =  a =2$.

Let $e_1$ be one of the four edges in the paths $P_{a+1}$ and $P_{b+1}$, and $e_2$ be one of the two edges incident to $x$ or $y$ in the path $P_{c+1}$. Then $\psi(e_1) \leq m-7$,  and $\psi(e_2) \leq m-6$. Hence,
\[
Mo_e(G) \leq 4(m-7)+2(m-6)+(m-6)(m-1) < m^2-m-28.
\]

\noindent {\bf Subcase 2.3.} $c \geq b \geq 3,  a =2$.

Let $e_1$ be one of the four edges incident to $x$ or $y$ in the paths $P_{b+1}$ and $P_{c+1}$, and $e_2$ is an edge in the path $P_{a+1}$. Then $\psi(e_1) \leq m-6$, and $\psi(e_2) \leq m-9$. Hence,
\[
Mo_e(G) \leq 4(m-6)+2(m-9)+(m-6)(m-1) < m^2-m-28.
\]

\noindent {\bf Case 3.} $c \geq b \geq 2 \geq  a =1$.

\noindent {\bf Subcase 3.1.} $c = b =3,  a =1$.

By Lemma \ref{32} and simple calculation, we have if $ m \geq 14$, then $Mo_e(B^5_m) > Mo_e(B^4_m)$;
if $ m=13$, then $Mo_e(B^3_m) = Mo_e(B^4_m) > Mo_e(B^5_m)$;
if $10 \leq m \leq 12$, then $Mo_e(B^4_m) > Mo_e(B^5_m)$;
if $ m=9$, then $Mo_e(B^3_m) = Mo_e(B^4_m) > Mo_e(B^5_m)$;
if $6 \leq m \leq 8$, $Mo_e(B^3_m) > Mo_e(B^4_m) > Mo_e(B^5_m)$;
if $m=5$, then $Mo_e(B^3_m) = Mo_e(B^4_m)$.

\noindent {\bf Subcase 3.2.} $c = 3, b =2,  a =1$.

This Subcase follows from Lemma \ref{34}.

\noindent {\bf Subcase 3.3.} $c \geq 4, b =2,  a =1$.

Let $e_1=xy$, $e_2$ be one of the four edges in the paths $P_{b+1}$ and $P_{c+1}$ incident to $x$ or $y$, and $e_3$ be one of the two edges in the middle of the path $P_{c+1}$. Then $\psi(e_1) = 0$, $\psi(e_2) \leq m-5$, and  $\psi(e_3) \leq m-6$. Thus,
\[
Mo_e(G) \leq 4(m-5)+2(m-6)+(m-7)(m-1) < m^2-m-28.
\]

\noindent {\bf Subcase 3.4.} $c \geq 4, b =3,  a =1$.

  Let $e_1=xy$, $e_2$ be one of the two edges in the path $P_{b+1}$  incident to $x$ or $y$, $e_3$ be one of the two edges in the path $P_{c+1}$  incident to $x$ or $y$, and $e_4$ be one of the two edges in the middle of the path $P_{c+1}$. Then $\psi(e_1) = 0$,  $\psi(e_2) \leq m-4$, $\psi(e_2) \leq m-5$, and  $\psi(e_4) \leq m-6$. Hence,
\[
Mo_e(G) \leq 2(m-4)+2(m-5)+2(m-6)+(m-7)(m-1) < m^2-m-28. \qedhere
\]
\end{proof}

The proof of the Theorem \ref{a} follows from Lemmas \ref{31} and \ref{35}.

\vspace{10mm}

\noindent {\bf Acknowledgement:} This work was partially supported  by the National Natural Science Foundation of China (Grant Nos. 12071194, 11571155 and 12071158).


\vspace{3mm}



\begin{thebibliography}{20}

\bibitem{AXK}  Y. Alizadeh, K. Xu, S. Klav\v{z}ar,  On the Mostar index of trees and product graphs.  Filomat  35 (2021) 4637--4643.

\bibitem{ACT} M. Arockiaraj, J. Clement,  N. Tratnik, Mostar indices of carbon nanostructures and circumscribed donut benzenoid systems. Int. J. Quantum Chem. 119 (2019) e26043.

\bibitem{DL}  K. Deng, S. Li, On the extremal values for the Mostar index of trees with given degree sequence. Appl. Math. Comput. 390 (2021) 11. 125598.

\bibitem{DL1}  K. Deng, S. Li, On the extremal Mostar indices of trees with a given segment sequence. Bull. Malays. Math. Sci. Soc. 45 (2021) 593--612.

\bibitem{DL2}  K. Deng, S. Li, Chemical trees with extremal Mostar index. MATCH Commun. Math. Comput. Chem.  85  (2021) 161--180.

\bibitem{DL3}  K. Deng, S. Li, Extremal catacondensed benzenoid with respect to the Mostar index. J. Math. Chem.  58  (2020) 1437--1465.

\bibitem{DoM}  T. Do\v{s}li\'c, I.  Martinjak, R. \v{S}krekovski, S. T. Spu\v{z}evi\'{c}, I.  Zubac, Mostar index. J. Math. Chem.  56  (2018) 2995--3013.

\bibitem{GXD}  F. Gao, K. Xu, T. Do\v{s}li\'c, On the difference between Mostar index and irregularity of graphs. Bull. Malays. Math. Sci. Soc. 44 (2021) 45. 905--926.

\bibitem{GAN}  A. Ghalavand, A. R. Ashrafi, M. H. Nezhaad, On Mostar and edge Mostar indices of graphs. J. Math. (2021) 6651220.

\bibitem{H}    O. C. Havare, Mostar index and edge Mostar index for some cycle related graphs. Rom. J. Math. Comput. Sci. 10 (2020) 53--66.

\bibitem{HX}  F. Hayat, S.-J. Xu, A lower bound on the Mostar index of tricyclic graphs. Filomat   (in press).

\bibitem{HX1}  F. Hayat, S.-J. Xu, Extremal results on the Mostar index of trees with parameters, Discrete Math. Algor. Appl. (2024), Doi: 10.1142/S1793830924500253.

\bibitem{HZ}  F. Hayat, B. Zhou, On cacti with large Mostar index. Filomat  33 (2019) 4865--4873.

\bibitem{HZ1} F. Hayat, B. Zhou, On Mostar index of trees with parameters. Filomat  33 (2019) 6453--6458.

\bibitem{HLM} S. Huang, S. Li, M. Zhang, On the extremal Mostar indices of hexagonal chains. MATCH Commun. Math. Comput. Chem.  84  (2020) 249--271.

\bibitem{IAI} M. Imran, S. Akhter,  Z. Iqbal,  Edge Mostar index of chemical structures and nanostructures using graph operations. Int. J. Quan. Chem. 120 (2020) e26259.

\bibitem{LD}  G. Liu, K. Deng, The maximum Mostar indices of unicyclic graphs  with given diameter. Appl. Math. Comput. 439 (2023) 127636.

\bibitem{LSX} H. Liu, L. Song, Q. Xiao, Z. Tang,   On edge Mostar index of graphs. Iranian J. Math. Chem. 11(2) (2020) 95--106.

\bibitem{Te} A. Tepeh,  Extremal bicyclic graphs with respect to Mostar index.  Appl. Math. Comput.  355  (2019) 319--324.

\bibitem{XZT}  Q. Xiao, M. Zeng, Z. Tang,  The hexagonal chains with the first three maximal Mostar indices. Discrete  Appl. Math. 288  (2020) 180--191.

\bibitem{XZT2}  Q. Xiao, M. Zeng, Z. Tang,  H. Deng, H. Hua,  Hexagonal chains with first three minimal Mostar indices. MATCH Commun. Math. Comput. Chem.  85  (2021) 47--61.

\end{thebibliography}
\end{document}